\documentclass[reqno]{amsart}

\usepackage[colorlinks=true, pdfstartview=FitV, linkcolor=blue,citecolor=red, urlcolor=blue]{hyperref}

\usepackage{amsmath,amsthm,amssymb}
\usepackage{color}
\usepackage{hyperref}
\usepackage{url}
\newcommand{\bburl}[1]{\textcolor{blue}{\url{#1}}}

\newtheorem{thm}{Theorem}[section]

\theoremstyle{definition}

\theoremstyle{definition}
\newtheorem{defi}[thm]{Definition}
\theoremstyle{remark}
\newtheorem{rem}[thm]{Remark}

\newcommand\be{\begin{equation}}
\newcommand\ee{\end{equation}}
\newcommand\bee{\begin{equation*}}
\newcommand\eee{\end{equation*}}
\newcommand\ben{\begin{enumerate}}
\newcommand\een{\end{enumerate}}

\newcommand{\res}{\text{Res}}

\newcommand{\A}{\ensuremath{\mathbb{A}}}
\newcommand{\R}{\ensuremath{\mathbb{R}}}

\newcommand{\Z}{\ensuremath{\mathbb{Z}}}
\newcommand{\Q}{\mathbb{Q}}

\numberwithin{equation}{section}
\DeclareMathOperator{\tr}{tr}
\DeclareMathOperator{\Gal}{Gal}
\newcommand{\G}{\ensuremath{\mathbb{G}}}

\title[On smoothing singularities of elliptic orbital integrals on GL($n$)]
{On smoothing singularities of elliptic orbital integrals on GL($n$) and Beyond Endoscopy}

\author[O.E. Gonz\'alez]{Oscar E. Gonz\'alez}
\email{\textcolor{blue}{\href{mailto:oscar.gonzalez3@upr.edu}{oscar.gonzalez3@upr.edu}}}
\address{Department of Mathematics, University of Puerto Rico, R\'io Piedras, PR 00931}

\author[C.H. Kwan]{Chung Hang Kwan}
\email{\textcolor{blue}{\href{mailto:kevinkwanch@gmail.com}
{kevinkwanch@gmail.com}}}
\address{Department of Mathematics, the Chinese University of Hong Kong, Shatin, Hong Kong }

\author[S.J. Miller]{Steven J. Miller}
\email{\textcolor{blue}{\href{mailto:sjm1@williams.edu}{sjm1@williams.edu}},  \textcolor{blue}{\href{Steven.Miller.MC.96@aya.yale.edu}{Steven.Miller.MC.96@aya.yale.edu}}}
\address{Department of Mathematics and Statistics, Williams College, Williamstown, MA 01267}

\author[R. Van Peski]{Roger Van Peski}
\email{\textcolor{blue}{\href{mailto:rpeski@princeton.edu}
{rpeski@princeton.edu}}}
\address{Department of Mathematics, Princeton University, Princeton, NJ 08544}

\author[T.A. Wong]{Tian An Wong}
\email{\textcolor{blue}{\href{mailto:twong@columbia.edu}
{t.wong@columbia.edu}}}
\address{Max Planck Institut f\"ur Mathematik, Vivatsgasse 7, Bonn, 53111}



\date{\today}

\begin{document}

\begin{abstract}
Recent work of Altu\u{g} continues the preliminary analysis of Langlands' Beyond Endoscopy proposal for $GL(2)$ by removing the contribution of the trivial representation by a Poisson summation formula. We show that Altu\u{g}'s method of smoothing real elliptic orbital integrals by an approximate functional equation extends to $GL(n)$. We also discuss the case of an arbitrary reductive group, and remaining obstructions for applying Poisson summation.
\end{abstract}

\maketitle

\tableofcontents

%%%%%%%%%%%%%%%%%%%%%%%%%%%%%%%%%%%%%%%%%%%%%%%%%%%%%%%%%%%%%%%%%%%%%%%%%%%%%%%%%%%%%%%%%%%%%%%%%%%%%%%%%%%%%%%%%%%%%%%%%%%%%%%%%%%%%%%%%%%%
%%%%%%%%%%%%%%%%%%%%%%%%%%%%%%%%%%%%%%%%%%%%%%%%%%%%%%%%%%%%%%%%%%%%%%%%%%%%%%%%%%%%%%%%%%%%%%%%%%%%%%%%%%%%%%%%%%%%%%%%%%%%%%%%%%%%%%%%%%%%
%%%%%%%%%%%%%%%%%%%%%%%%%%%%%%%%%%%%%%%%%%%%%%%%%%%%%%%%%%%%%%%%%%%%%%%%%%%%%%%%%%%%%%%%%%%%%%%%%%%%%%%%%%%%%%%%%%%%%%%%%%%%%%%%%%%%%%%%%%%%
\section{Introduction}

%%%%%%%%%%%%%%%%%%%%%%%%%%%%%%%%%%%%%%%%%%%%%%%%%%%%%%%%%%%%%%%%%%%%%%
%%%%%%%%%%%%%%%%%%%%%%%%%%%%%%%%%%%%%%%%%%%%%%%%%%%%%%%%%%%%%%%%%%%%%%
%%%%%%%%%%%%%%%%%%%%%%%%%%%%%%%%%%%%%%%%%%%%%%%%%%%%%%%%%%%%%%%%%%%%%%
\subsection{Overview of Beyond Endoscopy}

One of the key conjectures of the Langlands Program is the Functoriality Conjecture: given two reductive groups $G'$ and $G$, and an $L$-homomorphism of the associated $L$-group $^{L} G'$ to $^{L} G$, one expects a transfer of automorphic forms on $G'$ to automorphic forms on $G$. Most cases of functoriality known today fall under the banner of endoscopy, that is, where $G'$ is an endoscopic group of $G$. The problem of endoscopy is addressed by the stable trace formula, recently made unconditional by Ng\^{o}'s solution of the Fundamental Lemma \cite{Ngo}. Anticipating this, Langlands proposed a new strategy to attack the general case, referred to as Beyond Endoscopy \cite{Lan1}.

If an automorphic form $\pi$ on $G$ is a functorial transfer from a smaller $G'$, then one expects the $L$-function $L(s,\pi,r)$ to have a pole at $s = 1$ for some representation $r$ of $^{L}G$. In particular, the order of $L(s,\pi,r)$ at $s=1$, which we denote by $m_r(\pi)$, should be nonzero if and only if $\pi$ is a transfer. Langlands' idea then is to weight the spectral terms in the stable trace formula by $m_r(\pi)$, resulting in a trace formula whose spectral side detects only $\pi$ for which the $m_r(\pi)$ is nonzero.

Since in general $L(s,\pi,r)$ is not a priori defined at $s=1$, we account for the weight factor by taking the residue at $s=1$ of the logarithmic derivative of $L(s,\pi,r)$.\footnote{Another possibility is to take the residue of $L(s,\pi,r)$ itself at $s=1$, in which case $m_r(\pi)$ is more complicated.} This should lead to an $r$-trace formula,
\be
\label{rtrace}
S^r_\text{cusp}(f) = \lim_{N\to \infty }\frac{1}{|V_N|}\sum_{v\in V_N}\log (q_v) S^1_\text{cusp}(f^r_v),
\ee
where $S^1_\text{cusp}$ represents the usual stable trace formula, $V_N$ is a finite set of valuations of the global field $F$, and $q_v$ the order of the residue field of $F_v$ which is less than $N$ (see \cite[\S2]{Art} for details).

Following Arthur \cite{Art}, the stable distribution should have a decomposition
\be
S^r_\text{cusp}(f) = \sum_{G'} \iota(r,G')\hat{P}^{\tilde{G}'}_\text{cusp}(f'),
\ee
where $\hat{P}^{\tilde{G}'}_\text{cusp}(f')$ are called primitive stable distributions on elliptic `beyond endoscopic' groups $G'$, and by primitive one means the spectral contribution to the stable trace formula of tempered, cuspidal automorphic representations that are not functorial images from some smaller group. These primitive distributions, giving a new {\em primitive} trace formula, are to be defined inductively, and one hopes to establish these from the $r$-trace formula.

As is usual with trace formulae, one would like both the $r$-trace formula and the primitive trace formula to be an identity of spectral and geometric sides. But since one only wants tempered automorphic representations to contribute to \eqref{rtrace}, one has to first remove the contribution of the nontempered representations. Inspired by work of Ng\^o, a suggestion was put forth in \cite{FLN} to apply Poisson summation to the elliptic contribution, over a linear space called the Steinberg-Hitchin base. Assuming the existence of a such a summation, they show that the dominant term should cancel with the contribution of the trivial representation, which can be viewed as the most nontempered representation.

In related work, the recent thesis of Altu\u{g} completes the preliminary analysis carried out in \cite{Lan1} for $GL(2)$ and the standard representation (\cite{Alt1}, see also \cite{Alt2}). Working over $\Q$, and restricting ramification to the infinite prime, Altu\u{g} applies a modified form of the Poisson summation described in \cite{FLN}  by expressing the volume factors as values of Hecke $L$-functions and a strategic application of the approximate functional equation, the singularities of the archimedean orbital integral can be overcome, which is necessary in order to apply Poisson summation.

By a detailed analysis, Altu\u{g} shows that not only does the trivial representation contribute to the dominant term of the dual sum, but also the continuous spectral term associated to the nontrivial Weyl element of $GL(2)$. Based on this, Arthur outlines in \cite{Art} a list of problems to be addressed in order to establish the primitive trace formula.

%%%%%%%%%%%%%%%%%%%%%%%%%%%%%%%%%%%%%%%%%%%%%%%%%%%%%%%%%%%%%%%%%%%%%%
%%%%%%%%%%%%%%%%%%%%%%%%%%%%%%%%%%%%%%%%%%%%%%%%%%%%%%%%%%%%%%%%%%%%%%
%%%%%%%%%%%%%%%%%%%%%%%%%%%%%%%%%%%%%%%%%%%%%%%%%%%%%%%%%%%%%%%%%%%%%%
\subsection{Main result}

In this paper, we study to what extent Altu\u{g}'s method generalizes to a more general reductive group, focusing on Problem III discussed in \cite{Art}. In particular, we show that Altu\u{g}'s use of the approximate functional equation to smooth the singularities of real orbital integrals can be generalized to $GL(n)$. In particular, we apply Altu\u{g}'s method \cite{Alt1} to study the elliptic part of the trace formula of $G=GL(n)$,
\be
\sum_{\gamma \text{ ell}} \operatorname{meas}(\gamma)  \int_{G_{\gamma}(\A) \backslash G(\A)} f(g^{-1} \gamma g) dg,
\ee
where the sum is taken over elliptic conjugacy classes of $G(\Q)$. (See Section \ref{prelim} for precise definitions.) Choosing test functions as in \S2.1, we rewrite it as
\be
\label{eqn finalexpr}
\sum_{\pm p^{k}} \sum_{\tr(\gamma),\ldots, \tr(\gamma^{n-1})}
\frac{1}{|s_{\gamma}|}L(1,\sigma_{E})\theta_{\infty}^\pm(\gamma) \prod_{q}\operatorname{Orb}(f_{q};\gamma),
\ee
where $E$ is the extension of $\Q$ defined by the elliptic element $\gamma$, $\sigma_E$ the Galois representation appearing in the factorization $\zeta_{E}(s)=\zeta_{\Q}(s)L(s,\sigma_{E})$, and the product is taken over all primes $q$ of $\Q$. Here the $L$-value represents the global volume term meas$(\gamma)$ seen before.

We show that the approximate functional equation can be used again to smooth the singularities of the archimedean orbital integrals, generalizing \cite[Proposition 4.1]{Alt1}. The main result is the following.

\begin{thm}
\label{mainthm}
Let $\theta^\pm_ {\infty}(\gamma)$ be defined as in \eqref{thetapm}, $\phi$ any Schwartz function on $\mathbb{R}$, and $\alpha>0$.
Then the function defined by
\be
f(x_{1},\ldots, x_{n-1}) =  \theta_{\infty}^{\pm}(x_{1},\ldots, x_{n-1})\phi(\lvert D(x_{1},\ldots, x_{n-1})\rvert^{-\alpha})
\ee
is smooth.
\end{thm}

\noindent In particular, we will take $\phi$ to be the cutoff functions $V_s$ and $V_{1-s}$ in the approximate functional equation for $L(1,\sigma_E)$, so that the product
	\be
	L(1,\sigma_E)\theta_\infty^\pm(\gamma)
    \ee
is such that the singularities of the archimedean orbital integral, which lie along the vanishing set of the discriminant of $\gamma$, can be controlled. For the cases $GL(n)$ with $n>3$, we require Artin's conjecture as a simplifying assumption.

This result represents a first step towards establishing \eqref{rtrace} for higher rank groups. In particular, in order to apply Poisson summation over the Steinberg-Hitchin base, the singularities of the real orbital integrals must be addressed. Regarding this point, the method of the approximate functional equation has been the most successful up to now. While our application of the approximate functional equation is not completely analogous to Altu\u{g}'s (see Section \ref{sec6} for more detailed discussion), we expect that the techniques and considerations of this paper will be valuable to further attempts to generalize Altu\u{g}'s method.

%%%%%%%%%%%%%%%%%%%%%%%%%%%%%%%%%%%%%%%%%%%%%%%%%%%%%%%%%%%%%%%%%%%%%%
%%%%%%%%%%%%%%%%%%%%%%%%%%%%%%%%%%%%%%%%%%%%%%%%%%%%%%%%%%%%%%%%%%%%%%
%%%%%%%%%%%%%%%%%%%%%%%%%%%%%%%%%%%%%%%%%%%%%%%%%%%%%%%%%%%%%%%%%%%%%%

\subsection{Outline}

This paper is organized as follows. In Section \ref{prelim}, we introduce the necessary definitions and notation, and using the class number formula arrive at \eqref{eqn finalexpr}. In Section \ref{AFE}, we introduce the approximate functional equation for Artin $L$-functions. In Section \ref{Smoothing}, we describe the characterization of orbital integrals on real reductive groups, and prove Theorem \ref{mainthm}. Finally, in Section \ref{Directions} we give indications on how our analysis can be generalized to general reductive groups, using work of Ono and Shyr on Tamagawa numbers of algebraic tori. In Section \ref{sec6}, we explain how our theorem fits into the framework of generalizing Altu\u{g}'s analysis to $GL(n)$, and the remaining obstructions for applying Poisson summation.

\subsection{Acknowledgments}

The authors were supported by NSF Grant DMS1347804 and Williams College; the third named author was additionally supported by NSF grants DMS1265673 and DMS1561945. We thank James Arthur, Jasmin Matz, and Tasho Kaletha for helpful discussions, and S. Ali Altu\u{g} and Bill Casselman for comments on preliminary version of this paper. We also thank the referee for a careful reading of the paper and helpful suggestions.

%%%%%%%%%%%%%%%%%%%%%%%%%%%%%%%%%%%%%%%%%%%%%%%%%%%%%%%%%%%%%%%%%%%%%%%%%%%%%%%%%%%%%%%%%%%%%%%%%%%%%%%%%%%%%%%%%%%%%%%%%%%%%%%%%%%%%%%%%%%%
%%%%%%%%%%%%%%%%%%%%%%%%%%%%%%%%%%%%%%%%%%%%%%%%%%%%%%%%%%%%%%%%%%%%%%%%%%%%%%%%%%%%%%%%%%%%%%%%%%%%%%%%%%%%%%%%%%%%%%%%%%%%%%%%%%%%%%%%%%%%
%%%%%%%%%%%%%%%%%%%%%%%%%%%%%%%%%%%%%%%%%%%%%%%%%%%%%%%%%%%%%%%%%%%%%%%%%%%%%%%%%%%%%%%%%%%%%%%%%%%%%%%%%%%%%%%%%%%%%%%%%%%%%%%%%%%%%%%%%%%%
% Section: Preliminaries
%%%%%%%%%%%%%%%%%%%%%%%%%%%%%%%%%%%%%%%%%%%%%%%%%%%%%%%%%%%%%%%%%%%%%%%%%%%%%%%%%%%%%%%%%%%%%%%%%%%%%%%%%%%%%%%%%%%%%%%%%%%%%%%%%%%%%%%%%%%%
%%%%%%%%%%%%%%%%%%%%%%%%%%%%%%%%%%%%%%%%%%%%%%%%%%%%%%%%%%%%%%%%%%%%%%%%%%%%%%%%%%%%%%%%%%%%%%%%%%%%%%%%%%%%%%%%%%%%%%%%%%%%%%%%%%%%%%%%%%%%
%%%%%%%%%%%%%%%%%%%%%%%%%%%%%%%%%%%%%%%%%%%%%%%%%%%%%%%%%%%%%%%%%%%%%%%%%%%%%%%%%%%%%%%%%%%%%%%%%%%%%%%%%%%%%%%%%%%%%%%%%%%%%%%%%%%%%%%%%%%%
\section{Preliminaries}\label{prelim}

%%%%%%%%%%%%%%%%%%%%%%%%%%%%%%%%%%%%%%%%%%%%%%%%%%%%%%%%%%%%%%%%%%%%%%
%%%%%%%%%%%%%%%%%%%%%%%%%%%%%%%%%%%%%%%%%%%%%%%%%%%%%%%%%%%%%%%%%%%%%%
%%%%%%%%%%%%%%%%%%%%%%%%%%%%%%%%%%%%%%%%%%%%%%%%%%%%%%%%%%%%%%%%%%%%%%
\subsection{Notation}

We follow closely the setting of \cite{Lan1} and \cite{Alt1}. Let $G = GL(n)$ and $\A = \A_{\Q}$ be the ring of adeles of $\Q$. Denote by ${v}$ any valuation of $\Q$, $q$ any finite prime, and $p$ a fixed prime.

An element $\gamma\in G(\Q)$ is said to be \emph{elliptic} over $\Q$ if its  characteristic polynomial is irreducible over $\Q$. Throughout this paper, the word `elliptic' will be reserved for `elliptic over $\Q$'.
Let $Z_{+}$ be the set of all matrices in the center of $G(\mathbb{R})$ with positive entries, and $G_{\gamma}$ be the centralizer of
$\gamma$ in $G$.

The discriminant of $\gamma$ is given by
\be
D_\gamma = \prod_{i<j}(\gamma_{i}-\gamma_{j})^2,
\ee
where the $\gamma_{i}$'s are the distinct eigenvalues of $\gamma$. An elliptic element $\gamma$ defines, by its characteristic polynomial, a degree $n$ extension $E$ of $\Q$, such that
\be
D_{\gamma} = s_{\gamma}^2 D_{E}.
\ee
for some integer $s_\gamma$, and $D_E$ is the discriminant of $E$.

\begin{defi}
Now let $f=\prod_v f_v$ be a function in $C_{c}^{\infty}(G(\mathbb{A}))$. Define the global volume term
\be
\label{eqn meas}
 \operatorname{meas}(\gamma) =  \operatorname{meas}(Z_{+} G_{\gamma}(\Q) \backslash G_{\gamma}(\A))
\ee
and the orbital integral
\be
\label{eqn orb}
    \text{Orb}(f_{v};\gamma) =  \int_{G_{\gamma}(\Q_{v}) \backslash G(\Q_{v})} f_{v}(g^{-1} \gamma g) d{g}_{v}.
\ee
\end{defi}

Then the elliptic part of the Arthur-Selberg trace formula refers to
\be\label{eqn trace}
\sum_{\gamma \text{ ell}} \operatorname{meas}(\gamma)  \int_{G_{\gamma}(\A) \backslash G(\A)} f(g^{-1} \gamma g) dg
 = \
\sum_{\gamma \text{ ell}} \operatorname{meas}(\gamma) \prod_{v} \text{Orb}(f_{v};\gamma),
\ee
where the sum is understood to be over representatives $\gamma$ of elliptic conjugacy classes in $G(\Q)$. Note that since $G=GL(n)$, these orbital integrals are in fact stable distributions.

\begin{defi}
Fix a prime $p$ and integer $k\geq1$. The local test functions $f_v\in C_c^{\infty}(\mathbb{Q}_v)$ are chosen as follows: at the archimedean place, choose $f_{\infty}\in C_{c}^{\infty}(Z_+\backslash G(\mathbb{R}))$ such that its orbital integrals are compactly supported; at finite places, choose $f_q$ and $f_p^k$ such that $f$ is supported on the set $\gamma \in G(\mathbb{Z})$ with $|\det(\gamma)|=p^{k}$. Finally, we define $f^{p,k}$ by
\be
\label{eqn testfunc}
f^{p,k} :=  f_{\infty}\cdot f_p^{p,k}\cdot\prod_{\substack{q\neq p}}f_q^{p,k}.
\ee
\end{defi}

\begin{rem}
Note however, that this condition on the determinant is not necessary for our analysis, we only impose it to reflect more closely the setting of \cite[\S2]{Lan1} and \cite{Alt1}, where $f_p$ is chosen so that $\tr(\pi(f_p))=\tr(\text{Sym}^k(A(\pi_p)))$, where $A(\pi_p)$ is the Satake parameter of $\pi$ at $p$, and the rest meaning we only allow ramification at infinity.
\end{rem}

%%%%%%%%%%%%%%%%%%%%%%%%%%%%%%%%%%%%%%%%%%%%%%%%%%%%%%%%%%%%%%%%%%%%%%
%%%%%%%%%%%%%%%%%%%%%%%%%%%%%%%%%%%%%%%%%%%%%%%%%%%%%%%%%%%%%%%%%%%%%%
%%%%%%%%%%%%%%%%%%%%%%%%%%%%%%%%%%%%%%%%%%%%%%%%%%%%%%%%%%%%%%%%%%%%%%
%%%%%%%%%%%%%%%%%%%%%%%%%%%%%%%%%%%%%%%%%%%%%%%%%%%%%%%%%%%%
% subsection: Class Number Formula and Measures
%%%%%%%%%%%%%%%%%%%%%%%%%%%%%%%%%%%%%%%%%%%%%%%%%%%%%%%%%%%%

\subsection{Class number formula and measures}

Denote by $\mathbb{A}_{E}$ the adele ring of $E$, and $I_{E}=\mathbb{A}_{E}^{\times}$ the ideles of $E$. Let $|\,\cdot\,|_{v}$ be the normalized absolute value on the completion $E_{v}$ and $|\,\cdot\,|_{\mathbb{A}_{E}}:I_{E}\rightarrow\R^{\times}$ be the absolute value defined by
\be
|x|_{\mathbb{A}_{E}} =  \prod_{v}|x_{v}|_{v},
\ee
where $x=(x_{v})$. Here $|\,.\,|_{\mathbb{A}_{E}}$ is a group homomorphism and we define the norm-one idele group to be its kernel, denoted $I_{E}^{1}$, and $E^{\times}\backslash I_{E}^{1}$ the norm-one idele class group.

\begin{rem}
The measures and test functions in \eqref{eqn meas} and \eqref{eqn orb} are to be chosen analogously as in \cite{Alt1}. We mention the choices once again here. We first describe the choice of measure on $G$. At any finite prime $q$, we choose any Haar measure on $G(\mathbb{Q}_q)$ giving measure $1$ on $G(\mathbb{Z}_q)$, and the same with $G_\gamma(\Q_q)$; at infinity, we choose any Haar measure on $G(\R)$. In keeping with \cite[p.1797]{Alt1}, we note that there are more natural ways to normalize measures, but we make this choice so as to remain consistent with the hypotheses of \cite{Lan1,Alt1}.
\end{rem}

Let $\gamma$ be an elliptic element in $G(\Q)$. Recall that it defines a degree $n$ extension $E$ of $\Q$. It was observed by Langlands \cite[Equation (19)]{Lan1} that
\be
Z_{+}G_{\gamma}(\Q)\backslash G_{\gamma}(\A) = Z_{+}E^{\times}\backslash I_{E} = E^{\times}\backslash I_{E}^{1}.
\ee
Choosing the measures on $Z_{+}G_{\gamma}(\Q)\backslash G_{\gamma}(\A)$ as such, we also require the measure on $E^{\times}\backslash I_{E}^{1}$ to be such that
\be
\operatorname{meas}(Z_{+}G_{\gamma}(\Q)\backslash G_{\gamma}(\A)) = \operatorname{meas}(E^{\times}\backslash I_{E}^{1}).
\ee
so that we have by \cite[Theorem 4.3.2]{T}, the relation
\be\label{eqn tate}
\operatorname{meas}(E^{\times}\backslash I_{E}^{1}) = \frac{2^{r_1}(2\pi)^{r_2} h_{E} R_{E}}{w_{E}},
\ee
where $h_{E}$ is the class number of $E$, $R_{E}$ is the regulator of $E$, $r_{1}$ is the number of real embeddings of $E$ and $r_{2}$ is the number of pairs of complex embeddings of $E$, and $w_{E}$ is the number of roots of unity in $E$.

For a number field $E$, the Dedekind zeta function of $E$ is defined by
\be
\zeta_{E}(s) =  \sum_{\mathfrak{a}} \frac{1}{N_{E/\mathbb{Q}}(\mathfrak{a})^s} = \prod_{\mathfrak{p}}\frac{1}{1-N_{E/\mathbb{Q}}(\mathfrak{p})^s}
\ee
  for $\Re s>1$, where the sum is over all non-zero integral ideals of $E$ and the product is over all prime ideals of $E$.

Hecke showed that the Dedekind zeta function $\zeta_E(s)$ can be analytically continued to the whole complex plane except having a simple pole at $s=1$. Recall the classical Dirichlet class number formula
\be\label{eqn cnf}
\res_{s=1} \zeta_{E}(s) =
\frac{2^{r_1}(2\pi)^{r_2} h_{E} R_{E}}{w_{E}\sqrt{|D_E|}},
\ee
which will be related to the idele class group via \eqref{eqn tate}.

Here again $E$ is a number field. Let $\rho: \Gal(E/\mathbb{Q})\rightarrow \text{GL}(d,\mathbb{C})$ be a finite dimensional Galois representation. The Artin $L$-function associated to $\rho$ is given in terms of the Euler product:
\begin{equation}\label{eqn artin}
    L(s,\rho) =  \prod_{p} \det(I-\rho(\text{Fr}_{p})
    p^{-s})^{-1},
\end{equation}
  where $\text{Fr}_{p}$ is the Frobenius element in $\Gal(E/\Q)$.

In general, $E/\Q$ is not necessarily an Galois extension. In this case, the Artin $L$-function is obtained in the following manner: let $E^\mathrm{Gal}$ be the Galois closure of $E$ over $\Q$, $G=\Gal(E^\mathrm{Gal}/\Q)$, $H=\Gal(E^\mathrm{Gal}/E)$,
and $\rho_{G/H}$ be the permutation representation of $G$ on $G/H$. Then $\zeta_{E}(s)=L(s,\rho_{G/H})$. The representation can in turn be decomposed into irreducible representations $\rho_{i}$'s on $G$, i.e., $\rho_{G/H}= \oplus n_{i}\rho_{i}$, $n_{i} \ge 1$ are integers, and we have the factorization

\begin{equation}
\zeta_{E}(s) = \prod_{i} L(s,\rho_{i})^{n_{i}}. 
\end{equation}

Note that $\zeta_{E}(s)$ has a simple pole at $s=1$. We can also write $\zeta_{E}(s) = \zeta_{\mathbb{Q}}(s)L(s,\sigma_{E})$ for some representation $\sigma_{E}$. For more discussions on Artin $L$ functions, we refer the reader to \cite{IK, Mur, Neu}.

By equations \eqref{eqn tate} and \eqref{eqn cnf}, we have
\be
\operatorname{meas}(E^{\times}\backslash I_{E}^{1}) = \sqrt{|D_{E}|}\,\,L(1,\sigma_{E}),
\ee
and also
\be\label{CNFA}
\operatorname{meas}(Z_{+}G_{\gamma}(\Q)\backslash G_{\gamma}(\A)) = \sqrt{|D_{E}|}\,\,L(1,\sigma_{E}).
\ee

This is the form of the class number formula we shall use.

%%%%%%%%%%%%%%%%%%%%%%%%%%%%%%%%%%%%%%%%%%%%%%%%%%%%%%%%%%%%%%%%%%%%%%
%%%%%%%%%%%%%%%%%%%%%%%%%%%%%%%%%%%%%%%%%%%%%%%%%%%%%%%%%%%%%%%%%%%%%%
%%%%%%%%%%%%%%%%%%%%%%%%%%%%%%%%%%%%%%%%%%%%%%%%%%%%%%%%%%%%%%%%%%%%%%
\subsection{Rewriting the elliptic term}

By the choice of test function $f$ , the right hand side of \eqref{eqn trace} is non-zero if and only if $|\det(\gamma)|_{q}=1$ for any finite prime $q\neq p$ and $|\det(\gamma)|_{p}=p^{-k}$. Therefore $\det(\gamma)\in \Z_{q}$ for any finite prime $q$ and $\det(\gamma)\in\Z$. Moreover, $|\det(\gamma)|=p^{k}$.

We parametrize $\gamma\in G(\Q)$ by the coefficients of its characteristic polynomial,
\be
X^n-a_1X^{n-1}+\dots+(-1)^{n}a_n
\ee
where
\be
\label{eq:ParamGamma}
(a_1,\dots,a_n) = (\tr(\gamma),\ldots,\tr(\gamma^{n-1}),\det(\gamma)).
\ee
By \eqref{eq:ParamGamma} and $\det(\gamma)=\pm p^k$, we write
\be
\label{thetapm}
\theta_{\infty}^{\pm}(\gamma) =\ \theta_{\infty}^\pm(a_1,\ldots,a_{n-1},\pm p^{k}),
\ee
where $\theta_{\infty}^\pm(\gamma)=|D_{\gamma}|^{1/2}\operatorname{Orb}(f_{\infty};\gamma)$, since $\theta_\infty$ is invariant under conjugation, therefore we can consider $\theta_{\infty}^{\pm}(\gamma)$ as a function on $\R^{n-1}$.

The right hand side of \eqref{eqn trace} then becomes
\be\label{eqn finalexpr2}
\sum_{\pm p^{k}} \sideset{}{'}\sum
\frac{1}{|s_{\gamma}|}L(1,\sigma_{E})\theta_{\infty}^\pm(\gamma) \prod_{q}\operatorname{Orb}(f_{q}^{p,k};\gamma),
\ee
where the inner sum is taken over those $(a_{1},\ldots, a_{n-1})$ over $\Z^{n-1}$ corresponding to elliptic elements.

\begin{rem}
The discriminant that appears in the class number formula is that of the number field, whereas the discriminant that appears in the orbital integral is that of the characteristic polynomial. We will see from the proof of Theorem \ref{smoothe} that the factor $|s_{\gamma}|$ does not play a significant role in the issue of smoothness. 
\end{rem}

%%%%%%%%%%%%%%%%%%%%%%%%%%%%%%%%%%%%%%%%%%%%%%%%%%%%%%%%%%%%%%%%%%%%%%%%%%%%%%%%%%%%%%%%%%%%%%%%%%%%%%%%%%%%%%%%%%%%%%%%%%%%%%%%%%%%%%%%%%%%
%%%%%%%%%%%%%%%%%%%%%%%%%%%%%%%%%%%%%%%%%%%%%%%%%%%%%%%%%%%%%%%%%%%%%%%%%%%%%%%%%%%%%%%%%%%%%%%%%%%%%%%%%%%%%%%%%%%%%%%%%%%%%%%%%%%%%%%%%%%%
%%%%%%%%%%%%%%%%%%%%%%%%%%%%%%%%%%%%%%%%%%%%%%%%%%%%%%%%%%%%%%%%%%%%%%%%%%%%%%%%%%%%%%%%%%%%%%%%%%%%%%%%%%%%%%%%%%%%%%%%%%%%%%%%%%%%%%%%%%%%
% Section: AFE
%%%%%%%%%%%%%%%%%%%%%%%%%%%%%%%%%%%%%%%%%%%%%%%%%%%%%%%%%%%%%%%%%%%%%%%%%%%%%%%%%%%%%%%%%%%%%%%%%%%%%%%%%%%%%%%%%%%%%%%%%%%%%%%%%%%%%%%%%%%%
%%%%%%%%%%%%%%%%%%%%%%%%%%%%%%%%%%%%%%%%%%%%%%%%%%%%%%%%%%%%%%%%%%%%%%%%%%%%%%%%%%%%%%%%%%%%%%%%%%%%%%%%%%%%%%%%%%%%%%%%%%%%%%%%%%%%%%%%%%%%
%%%%%%%%%%%%%%%%%%%%%%%%%%%%%%%%%%%%%%%%%%%%%%%%%%%%%%%%%%%%%%%%%%%%%%%%%%%%%%%%%%%%%%%%%%%%%%%%%%%%%%%%%%%%%%%%%%%%%%%%%%%%%%%%%%%%%%%%%%%%
\section{The approximate functional equation}\label{AFE}

In this section, we introduce the approximate functional equation of Artin $L$-functions, and show that the cutoff functions $V_{s}$ and $V_{1-s}$ in it are smooth and have good decay. For convenience, we assume the Artin conjecture, i.e., any Artin $L$-function attached to an irreducible, non-trivial and Galois representation of a number field is entire. Without the assumption, the following theorem is still valid, except that there would be additional terms accounting for possible contributions of poles of $L(s,\rho)$ along the lines $\Re(s)=0,1,$ which can be explicitly given.

\begin{thm}
\label{AFEartin}
Let $L(s,\rho)$ be the Artin $L$-function associated to an irreducible, non-trivial, Galois representation $\rho$ with conductor $q=q(\rho)$,
and assume the Artin Conjecture. Let $G(u)$ be any function which is holomorphic and bounded in the strip $-4 < \Re u < 4$, even, and normalized by $G(0)=1$. Let
$X>0$ be a parameter to be chosen. Then for $s$ in the strip $0 \leq \Re s \leq 1$, we have
\be
L(s,\rho) = \sum_{n} \frac{\lambda_{\rho}(n)}{n^s} V_{s}\left( \frac{n}{X\sqrt{q}}\right)
+\epsilon(s,\rho) \sum_{n} \frac{\bar{\lambda_{\rho}}(n)}{n^{1-s}}
V_{1-s}\left(\frac{nX}{\sqrt{q}}\right),
\ee
where
\begin{align}
\nonumber
L(s,\rho) &=  \sum_{n=1}^{\infty}\frac{\lambda_{\rho}(n)}{n^s}, \ \epsilon(s,\rho) =  \epsilon(\rho)q^{\frac{1}{2}-s}
\frac{\gamma(1-s,\rho)}{\gamma(s,\rho)};\\
\nonumber
V_{s}(y) &= \frac{1}{2\pi i} \int_{(3)} y^{-u} G(u) \frac{\gamma(s+u,\rho)}{\gamma(s,\rho)} \frac{du}{u}
\end{align}
and $\epsilon(\rho)$ is the root number of $L(s,\rho)$ and is a complex number with modulus $1$.
\end{thm}

\begin{proof}
See Theorem 5.3 of \cite{IK}.
\end{proof}

%%%%%%%%%%%%%%%%%%%%%%%%%%%%%%%%%%%%%%%%%%%%%%%%%%%%%%%%%%%%%%%%%%%%%%
%%%%%%%%%%%%%%%%%%%%%%%%%%%%%%%%%%%%%%%%%%%%%%%%%%%%%%%%%%%%%%%%%%%%%%
%%%%%%%%%%%%%%%%%%%%%%%%%%%%%%%%%%%%%%%%%%%%%%%%%%%%%%%%%%%%%%%%%%%%%%
\subsection{Gamma factors of Artin $L$-functions}

The gamma factor of the Artin $L$-function is a product of local gamma factors $\gamma_{v}(s,\rho)$ over infinite places $v$ of a number field $E$. Let $r_{1}$ and $r_{2}$ be as before, so that $r_{1}+2r_{2}=n$, where $n$ is the degree of the number field $E$.
  Let $\sigma_{v}$ be the Frobenius conjugacy class associated to the completion $E_v$. Then $\sigma_{v}$ is of order $2$ if $v$ is a real place, and trivial if $v$ is complex.
Hence, we have
\be
\gamma_{v}(s,\rho) = \begin{cases}
\medskip
\pi^{-ds/2}  \Gamma\left(\frac{s}{2}\right)^{d} \Gamma\left(\frac{s+1}{2}\right)^{d} &\text{if $v$ is a complex place}\\
\pi^{-ds/2}  \Gamma\left(\frac{s}{2}\right)^{d^{+}_v} \Gamma\left(\frac{s+1}{2}\right)^{d^{-}_v}  &\text{if $v$ is a real place,}
\end{cases}
\ee
where $d = \deg(\rho)$ is the dimension of $\rho$ and $d^{+}_v$, $d^{-}_v$ are the multiplicities, respectively, of the eigenvalue $+1$, $-1$ for $\rho(\sigma_{v})$.

Denote by  $d^{+}, d^{-}$ the sum over real $v$ of $d^{+}_{v}$ and $d^{-}_{v}$, respectively. The global gamma factor of the Artin $L$-function is the following:
\begin{align}
\gamma(s,\rho) & =    \prod_{v|\infty} \gamma_{v}(s,\rho)\nonumber\\
%& = \prod_{v\,\,{\rm real}}\left(\pi^{-ds/2}  \Gamma\left(\frac{s}{2}\right)^{d^{+}_v} \Gamma\left(\frac{s+1}{2}\right)^{d^{-}_v}\right)\prod_{v\,\,{\rm complex}}\left(\pi^{-ds/2}  \Gamma\left(\frac{s}{2}\right)^{d} \Gamma\left(\frac{s+1}{2}\right)^{d}\right)\nonumber\\
& = \pi^{-ds r_{1}/2}\,\,\Gamma\left(\frac{s}{2}\right)^{d^{+}} \Gamma\left(\frac{s+1}{2}\right)^{d^{-}} \pi^{-ds r_{2}/2}\,\,\Gamma\left(\frac{s}{2}\right)^{dr_{2}}
\Gamma\left(\frac{s+1}{2}\right)^{dr_{2}}\nonumber\\
& = \pi^{-ds(r_{1}+r_{2})/2}\ \Gamma\left(\frac{s}{2}\right)^{d^{+}+dr_{2}}\Gamma\left(\frac{s+1}{2}\right)^{d^{-}+dr_{2}}.
\end{align}

As before, $E$ is the degree $n$ number field defined by an elliptic element in $GL_{n}(\Q)$. We have the factorization $\zeta_{E}(s)$ $=$ $\zeta_{\mathbb{Q}}(s)L(s,\sigma_{E})$ in which the degree of the representation $\sigma_{E}$ is $n-1$. If $\sigma_{E}$ is a reducible representation, then $L(s,\sigma_{E})$ further factorizes into a product of $L$-functions of non-trivial irreducible representations of $\text{Gal}(E/\Q)$. 

If $n=2$, we have $\zeta_{E}(s) =  \zeta_{\Q}(s)L(s,\chi)$ for some non-trivial quadratic Dirichlet character $\chi$, where $\zeta_{\Q}$ is the Riemann zeta function. The further analysis is completed in \cite{Alt1}. It is a classical result, due to Brauer, stating that if $E/ \Q$ is an abelian extension, then $\zeta_{E}(s)$ is indeed a product of Dirichlet $L$-functions attached to distinct primitive characters and exactly one of the character is trivial.

Now suppose $n=3$ and $E$ is a cubic field. Then the degree of the Galois representation $\sigma_{E}$ is 2. Either $(r_{1},r_{2})=(1,1)$ or $(r_{1},r_{2})=(3,0)$. We consider both cases as follow.

\ben
\item
Case (1): $(r_{1},r_{2})=(1,1)$. The cubic extension $E/\Q$ is not Galois, and $\sigma_{E}$ is an irreducible representation.
\item
Case (2): $(r_{1},r_{2})=(3,0)$. The cubic extension is Galois, and $\sigma_{E}$ is a reducible representation, decomposing as $\chi$ and $\chi^{-1}$, where $\chi$ is the cubic character. The conjugacy class is order two and our character is order three, so we only get $d^{+} = 2$ and $d^{-} = 0$. Hence, we have the factorization: $\zeta_{E}(s) = \zeta_{\Q}(s)L(s,\chi)L(s,\chi^{-1})$, where $L(s,\chi)$ is the Hecke $L$-function associated to the cubic character $\chi$ defined by $E$. In fact, $L(s,\chi)L(s,\chi^{-1})$ can be written as the following Euler product
\begin{equation*}
    \prod_{\chi(p)=1} \frac{1}{(1-p^{-s})^2} \prod_{\chi(p)\neq 1} \frac{1}{1+p^{-s}+p^{-2s}}.
\end{equation*}
\een

One knows that the irreducible two-dimensional representation of $\Gal(E/\Q)=S_3$ corresponds to a modular form, hence $L(s,\sigma_{E})$ is entire in Case (1). Case (2) is simpler, and is known by class field theory. For $n>3$, there could be Artin $L$-functions associated to irreducible representations of degree greater than 2 in which assuming the Artin conjecture is necessary.

Having an explicit expression of the Euler product \eqref{eqn artin} just like the one above would be a crucial step for the combination of the volume term with the $q$-adic orbital integrals in order to form a `modified' $L$-function, such as in \cite[p.17]{Alt2}, that would be a weighted sum of $L$-functions associated to the quadratic residue symbol.

%%%%%%%%%%%%%%%%%%%%%%%%%%%%%%%%%%%%%%%%%%%%%%%%%%%%%%%%%%%%%%%%%%%%%%
%%%%%%%%%%%%%%%%%%%%%%%%%%%%%%%%%%%%%%%%%%%%%%%%%%%%%%%%%%%%%%%%%%%%%%
%%%%%%%%%%%%%%%%%%%%%%%%%%%%%%%%%%%%%%%%%%%%%%%%%%%%%%%%%%%%%%%%%%%%%%
\subsection{The cutoff functions}

We now examine the cutoff functions in the approximate functional equation of Artin $L$ functions for general Galois representation $\rho$ with degree $d$. From equation \eqref{eqn finalexpr2}, we only need the approximate functional equation for the case $s=1$. Hence in the following we restrict $s$ to be a real number in the range $[0,1]$.

\begin{thm}
Let $s$ be a real number in the range $[0,1]$. Then the functions $V_{s}$ and $V_{1-s}$ are smooth functions with the decay property that for any real number $m>3$, as $y\to\infty$, we have
\begin{equation}
V_{s}(y), V_{1-s}(y) \ll y^{-m}.
\end{equation}
\end{thm}
\begin{proof}
It is obvious that $V_{s}$ is a smooth function by differentiation under the integral sign. Next we consider its decay. The choice of the bounded holomorphic function $G$ in the approximate functional equation will not be important to us, hence we will simply assume that $G(u)$ is identically $1$. Then
\be
V_{s}(y)  =  \frac{1}{2\pi i  \gamma(s,\rho)}
\int_{(3)} y^{-u}
\frac{\Gamma\left(\frac{s+u}{2}\right)^{a}
\Gamma\left(\frac{s+u+1}{2}\right)^{b}}{u\pi^{d(s+u)(r_{1}+r_{2})/2}}  du,
\ee
where $a:=d^{+}+dr_{2}$ and $b:=d^{-}+dr_{2}$.

We will need the following form of Stirling's approximation. Let $u=\sigma+it$, $-\infty<\sigma_{1}\le \sigma\le \sigma_{2}<+\infty$ and $|t|\ge 1$.  Then we have
\be
\label{Stirling}
|\Gamma(u)|  =\   (2\pi)^{1/2} |t|^{\sigma - \frac{1}{2}} e^{-\pi|t|/2}(1+O(|t|^{-1})).
\ee
\noindent uniformly in  $-\infty<\sigma_{1}\le \sigma\le \sigma_{2}<+\infty$ as $|t|\to\infty$.

 We would like to do a contour shift from $(3)$ to $(m)$ by considering the rectangle of height $2T$ that is symmetric about the $x$-axis with vertical sides at $(3)$ and $(m)$. First we treat the integral along the top part of the contour $[3+iT,m+iT]$.  We remark that the implicit constants in the estimations below may depend on $s$ and $\rho$. The calculations for the bottom contour follows similarly and therefore will be omitted.

First, we have
\begin{align*}
&\int_{3+iT}^{m+iT} y^{-u} \frac{\Gamma\left(\frac{s+u}{2}\right)^{a}
	\Gamma\left(\frac{s+u+1}{2}\right)^{b}}{u\pi^{d(s+u)(r_{1}+r_{2})/2}}  du\\	
 \ll & \int_{3}^{m} y^{-\sigma}
\big|\Gamma\left(\frac{s+\sigma}{2} + i\frac{T}{2}\right)\big|^{a}
\big|\Gamma\left(\frac{s+\sigma+1}{2} + i\frac{T}{2}\right)\big|^{b}
\pi^{-d(s+\sigma)(r_{1}+r_{2})/2}  \frac{d\sigma}{\sqrt{\sigma^2+T^2}}.
\end{align*}
Then applying Stirling, this is
\[
\ll\int_{3}^{m} y^{-\sigma}
\left( \sqrt{2\pi}  \left(\frac{T}{2}\right)^{\frac{s+\sigma-1}{2}}  e^{-\pi T/4} \right)^{a} \left( \sqrt{2\pi} \left(\frac{T}{2}\right)^{\frac{s+\sigma}{2}}  e^{-\pi T/4} \right)^{b}
 \frac{\pi^{-d\sigma (r_{1}+r_{2})/2}}{\sqrt{\sigma^2+T^2}} d\sigma,
 \]
 and thus
 \begin{align}
&  \ll   T^{\frac{(a+b)s-a}{2}-1}e^{-\frac{(a+b)\pi T}{4}} \int_{3}^{m}  \big(\frac{T^{(a+b)/2}}{2^{(a+b)/2}\ y \pi^{d(r_{1}+r_{2})/2} }\big)^{\sigma} d\sigma \nonumber\\
%& \ll   T^{\frac{b}{2}-1}e^{-\frac{(a+b)\pi T}{4}}\frac{(\frac{T}{2})^{\frac{(a+b)m}{2}}}{y^m \log T}\notag\\
& \ll   \frac{T^{\frac{(a+b)m+b}{2}-1}}{y^m \log T} e^{-\frac{(a+b)\pi T}{4}}.
\end{align}

Next we treat the integral along the right vertical contour $[m-iT, m+iT]$. We first write
\be
\int_{m-iT}^{m+iT} y^{-u}
\frac{\Gamma\left(\frac{s+u}{2}\right)^{a}
\Gamma\left(\frac{s+u+1}{2}\right)^{b}}{u\pi^{d(s+u)(r_{1}+r_{2})/2}}  du
=
\int_{-T}^{T} y^{-m-it} \frac{\Gamma(\frac{s+m+it}{2})^a \Gamma(\frac{s+m+it+1}{2})^b}{(m+it)\pi^{d(s+m+it)(r_{1}+r_{2})/2}}  idt.
\ee
Then take absolute values and apply again Stirling's formula, we have
\begin{align}
&  y^{-m} \int_{-T}^{T} \frac{1}{\sqrt{m^2+t^2}} \big|\Gamma\left(\frac{s+m}{2} + i\frac{t}{2}\right)\big|^{a}
\big|\Gamma\left(\frac{s+m+1}{2} + i\frac{t}{2}\right)\big|^{b} dt \nonumber\\
&  \ll  y^{-m} + y^{-m} \int_{c}^{T} \frac{1}{\sqrt{m^2+t^2}}\big( \sqrt{2\pi}  \left(\frac{t}{2}\right)^{\frac{s+m-1}{2}}  e^{-\pi t/4} \big)^{a} \big( \sqrt{2\pi}  \left(\frac{t}{2}\right)^{\frac{s+m}{2}}  e^{-\pi t/4} \big)^{b} dt \nonumber\\
& \ll y^{-m}\big(1+  \int_{c}^{T} t^{(s+m)(a+b)/2-a/2-1} e^{-\pi t (a+b)/4} dt \big)
\end{align}
for some constant $c>0$ large enough.

We note that the integrand defining $V_s$ has no poles in the region $\Re{u}>3$ since $\Gamma(z)$ is holomorphic on $\Re{z}>0$. Therefore by Cauchy's Theorem, we have
\begin{align*}
&\int_{3-iT}^{3+iT} y^{-u}
\frac{\Gamma\left(\frac{s+u}{2}\right)^{a}
\Gamma\left(\frac{s+u+1}{2}\right)^{b}}{u\pi^{d(s+u)(r_{1}+r_{2})/2}}  du \\
& \ll y^{-m}\big(1+  \int_{c}^{T} t^{(s+m)(a+b)/2-a/2-1} e^{-\pi t (a+b)/4} dt \big) +\frac{T^{\frac{(a+b)m+b}{2}-1}}{y^m \log T} e^{-\frac{(a+b)\pi T}{4}}.\nonumber\\
\end{align*}
It is clear that the integral
\begin{equation}
\int_{c}^{\infty} t^{(s+m)(a+b)/2-a/2-1} e^{-\pi t (a+b)/4} dt
\end{equation}
converges. Therefore, letting $T\to\infty$, we have
\begin{equation}
\int_{(3)} y^{-u}
\frac{\Gamma\left(\frac{s+u}{2}\right)^{a}
\Gamma\left(\frac{s+1+u}{2}\right)^{b}}{u\pi^{d(s+u)(r_{1}+r_{2})/2}}  du \ll \frac{1}{y^{m}}
\end{equation}
and the result follows.

\end{proof}

%%%%%%%%%%%%%%%%%%%%%%%%%%%%%%%%%%%%%%%%%%%%%%%%%%%%%%%%%%%%%%%%%%%%%%%%%%%%%%%%%%%%%%%%%%%%%%%%%%%%%%%%%%%%%%%%%%%%%%%%%%%%%%%%%%%%%%%%%%%%
%%%%%%%%%%%%%%%%%%%%%%%%%%%%%%%%%%%%%%%%%%%%%%%%%%%%%%%%%%%%%%%%%%%%%%%%%%%%%%%%%%%%%%%%%%%%%%%%%%%%%%%%%%%%%%%%%%%%%%%%%%%%%%%%%%%%%%%%%%%%
%%%%%%%%%%%%%%%%%%%%%%%%%%%%%%%%%%%%%%%%%%%%%%%%%%%%%%%%%%%%%%%%%%%%%%%%%%%%%%%%%%%%%%%%%%%%%%%%%%%%%%%%%%%%%%%%%%%%%%%%%%%%%%%%%%%%%%%%%%%%
% Section: Smoothing
%%%%%%%%%%%%%%%%%%%%%%%%%%%%%%%%%%%%%%%%%%%%%%%%%%%%%%%%%%%%%%%%%%%%%%%%%%%%%%%%%%%%%%%%%%%%%%%%%%%%%%%%%%%%%%%%%%%%%%%%%%%%%%%%%%%%%%%%%%%%
%%%%%%%%%%%%%%%%%%%%%%%%%%%%%%%%%%%%%%%%%%%%%%%%%%%%%%%%%%%%%%%%%%%%%%%%%%%%%%%%%%%%%%%%%%%%%%%%%%%%%%%%%%%%%%%%%%%%%%%%%%%%%%%%%%%%%%%%%%%%
%%%%%%%%%%%%%%%%%%%%%%%%%%%%%%%%%%%%%%%%%%%%%%%%%%%%%%%%%%%%%%%%%%%%%%%%%%%%%%%%%%%%%%%%%%%%%%%%%%%%%%%%%%%%%%%%%%%%%%%%%%%%%%%%%%%%%%%%%%%%
\section{Smoothing of the real orbital integral}\label{Smoothing}

We follow the exposition of \cite[\S 4.2--4.3]{Alt2} and substitute the approximate functional equation of Artin $L$-function into equation \eqref{eqn finalexpr2}. This illustrates that the techniques used in \cite{Alt1, Alt2} to handle the singularities of the real orbital integral hold for $GL(n)$.

Suppose $\sigma_{E} \ = \ \rho_{1} \oplus \cdots  \oplus \rho_{t}$, where $\rho_{1},\dots, \rho_{t}$ are non-trivial irreducible representations on $\Gal(E^\mathrm{Gal}/\Q)$. Then by the absolute convergence of the series in the approximate functional equation, we may expand the product of Artin $L$-functions as follows

\begin{equation}
L(1,\sigma_{E})  =  \sum_{B\subset\{1,\ldots, t\}} \sum_{n_{1},\ldots,n_{t}} \prod_{i\in B} \frac{\lambda_{\rho_{i}}(n_{i})}{n_{i}} V_{1}^{(i)}\left(\frac{n_{i}}{X_{i}\sqrt{q_{i}}}\right) \prod_{i\not\in B} \epsilon(1,\rho_{i}) \overline{\lambda_{\rho_{i}}(n_{i})} V_{0}^{(i)}\left(\frac{n_{i} X_{i}}{\sqrt{q_{i}}}\right).
\end{equation}

\subsection{The singularities}
\label{sing}

We write the elliptic contribution as
\begin{align}
\sum_{\pm p^{k}} \sideset{}{'}\sum
&\left(\sum_{B\subset\{1,\ldots, t\}} \sum_{n_{1},\ldots,n_{t}} \prod_{i\in B} \frac{\lambda_{\rho_{i}}(n_{i})}{n_{i}} V_{1}^{(i)}\left(\frac{n_{i}}{X_{i}\sqrt{q_{i}}}\right) \prod_{i\not\in B} \epsilon(1,\rho_{i}) \overline{\lambda_{\rho_{i}}(n_{i})} V_{0}^{(i)}\left(\frac{n_{i} X_{i}}{\sqrt{q_{i}}}\right)\right)\nonumber\\
&\times\frac{1}{|s_{\gamma}|}\theta_{\infty}^\pm(\gamma) \prod_{q}\operatorname{Orb}(f_{q}^{p,k};\gamma).
\end{align}

Again by absolute convergence, we may switch the order of the summations,  yielding the following expression

\begin{align}\label{coeff}
\sum_{\pm p^{k}}\sum_{B\subset\{1,\ldots,t\}}\sum_{n_{1},\dots, n_{t}} \prod_{j\in B} \frac{1}{n_{j}} \sideset{}{'}\sum \bigg( \frac{1}{|s_{\gamma}|} \prod_{i\in B} \lambda_{\rho_{i}}(n_{i})\times \prod_{i\not\in B} \epsilon(1,\rho_{i}) \overline{\lambda_{\rho_{i}}(n_{i})}\ \times \nonumber\\
\prod_{i\in B } V_{1}^{(i)}\left(\frac{n_{i}}{X_{i}\sqrt{q_{i}}}\right) \times \prod_{i\not\in B} V_{0}^{(i)}\left(\frac{n_{i} X_{i}}{\sqrt{q_{i}}}\right) \theta_{\infty}^\pm(\gamma) \prod_{q}\operatorname{Orb}(f_{q}^{p,k};\gamma) \bigg)
\end{align}

\noindent and we would like to focus on the innermost sum $\sum^{'}$.

For $i=1,\ldots, t$, plug-in $X_{i}=|D_{E^\mathrm{Gal}}|^{c}$ with $0<c<1/2$ and the expressions of the conductors $q_{i}$ (c.f. pp.142 of \cite{IK}) are given by
\begin{equation}
q_{i}=|D_{E^\mathrm{Gal}}|^{d_{i}}N_{E^\mathrm{Gal}/\mathbb{Q}}(\mathfrak{f}_{i}),
\end{equation}
with $d_{i}$ being the degree of $\rho_{i}$ and  $\mathfrak{f}_{i}$ being the Artin conductor, which is a non-zero integral ideal of $E^\mathrm{Gal}$.

Referring the discussion preceding Section \ref{AFE}, we now turn to the smoothness of the functions
\begin{equation}\label{eqn prod1}
V_{1}^{(i)}\left( \frac{n_{i} }{|D_{E^\mathrm{Gal}}|^{\frac{d_{i}}{2}+c}\sqrt{N_{E^\mathrm{Gal}/\mathbb{Q}}(\mathfrak{f}_{i})}}\right)\theta_{\infty}^\pm(\gamma) ,
\end{equation}
and
\begin{equation}\label{eqn prod2}
V_{0}^{(i)}\left(\frac{n_{i}}{|D_{E^\mathrm{Gal}}|^{\frac{d_{i}}{2}-c}\sqrt{N_{E^\mathrm{Gal}/\mathbb{Q}}(\mathfrak{f}_{i})}}\right)\theta_{\infty}^\pm(\gamma).
\end{equation}

Since $D_{\gamma}$, $D_{E^\mathrm{Gal}}$ differ only by a non-zero integral multiple and $N_{E^\mathrm{Gal}/\mathbb{Q}}(\mathfrak{f}_{i})$ are  positive integers, the smoothing of the expressions \eqref{eqn prod1} and \eqref{eqn prod2} now follow as an immediate application of Theorem \ref{smoothe}. On the other hand, the behaviour of $V_0$ and $V_1$ are discussed in the previous section, so it remains to understand the singularities of $\theta_{\infty}^\pm(\gamma)$. Consider the real orbital integral $\operatorname{Orb}(f_{\infty};\gamma)$. Let $\gamma$ be a regular element in $G$, which is an $n$ by $n$ matrix with distinct eigenvalues, and $T_\mathrm{reg}$ be the set of all regular elements in $T=G_\gamma$. Also let $f_{\infty}\in C_{c}^{\infty}(Z_{+}\setminus G(\mathbb{R}))$.
Then it is well known that
 \be
 \theta_{\infty}^\pm(\gamma)=|D_{\gamma}|^{1/2}\operatorname{Orb}(f_{\infty};\gamma)
 \ee
 extends smoothly to $T_\mathrm{reg}$. This follows from the proof of Theorem 17.1 of \cite[p.145]{HC1}, and in particular Lemma 17.4 of \cite[p.147]{HC1} and Lemma 40 of \cite[p.491]{HC2}. Our function $\theta^\pm_\infty(\gamma)$ is related to the transform ${'F}_f^T$ of Harish-Chandra by a function of $\gamma$, which does not affect our analysis.

Notice that $z\in T_\mathrm{reg}$ if and only if $D_{\gamma}$ is nonzero. In other words, $\text{Orb}(f_{\infty};\gamma)$ may have singularities on $T\setminus T_\mathrm{reg}$. Depending on the root, one either has a removable singularity or the `jump relations' of Harish-Chandra along the walls of $T$, though we do not need their precise form here (c.f. \cite{Sh1,Sh2, Bou}). Indeed, given $f(x)g(1/D(x))$, with $g(x)$-Schwarz class, as long as the singularities of the function $f(x)$ is on the zero locus of $D(x)$ and are not worse than $x^{-\beta}$ for some $\beta\geq 0$, then the product function is smooth.\footnote{We thank S.A. Altu\u{g} for pointing this out to us.} Note also that for $p$-adic orbital integrals, a similar behavior may also occur, but again, we do not consider them in this paper.

\subsection{Smoothing}

The main theorem follows now from the preceding discussion.

\begin{thm}\label{smoothe}
Let $\phi$ be any Schwartz function on $\mathbb{R}$
and $\alpha$ be a positive constant.
Then the function defined by
\be
f(x_{1},\ldots, x_{n-1}) =\  \theta_{\infty}(x_{1},\ldots, x_{n-1})\phi(\lvert D(x_{1},\ldots, x_{n-1})\rvert^{-\alpha})
\ee
is smooth.
\end{thm}

\begin{proof}
 By the discussions precede Theorem \ref{smoothe} and the fact that $\phi$ is a Schwartz function, we have $f$ being a smooth function on $T_{\text{reg}}$.

Also, the singularities of the real orbital integral must lie on the zero locus of $D(x_{1},\ldots,x_{n-1})$ and they are not worse than $|x|^{-\beta}$ for some $\beta \ge 0$ ($\beta$ may depend on the choice of point of singularity). Now fix a singular point $a=(a_{1},\ldots, a_{n-1})\not\in T_{\text{reg}}$. As $x=(x_{1},\ldots, x_{n-1})\to a=(a_{1},\ldots, a_{n-1})$, we have  $|D(x)|^{-\alpha} \to \infty$ and
\begin{equation}
\phi(|D(x)|^{-\alpha}) \ \ll_{M} \ |D(x)|^{M\alpha}
\end{equation}
for any positive constant $M$. Then choose any $M >  (\beta+1)/ \alpha$, we have

\begin{align}
|f(x)| \ &\ll \ |x-a|^{-\beta}|D(x)|^{M\alpha } \ \le \ |x-a|^{-\beta}(|D(x)-Q(x-a)|+|Q(x-a)|)^{M\alpha} \nonumber\\
\ &\ll \ |x-a|^{-\beta} |D(x)-Q(x-a)|^{M\alpha} + ||Q||^{M\alpha}|x-a|^{M\alpha-\beta}\nonumber\\
\ &= \ |D(x)-Q(x-a)|^{M\alpha-\beta} \bigg(\frac{|D(x)-D(a)-Q(x-a)|}{|x-a|}\bigg)^{\beta}\nonumber\\
&\ \ \ + ||Q||^{M\alpha}|x-a|^{M\alpha-\beta}\nonumber,\\
\end{align}
where $Q$ is the Jacobi matrix of $D$ at $x=a$ and $||Q||$ denotes the matrix $2$-norm of $Q$. By the continuity and differentiability of $D$, we have the last expression tends to $0$ as $x\to a$. Therefore,
\begin{equation}
\lim_{x\to a} f(x) =  0
\end{equation}
and we can redefine $f(a)=0$.

Let $h\neq 0$. By mean value theorem, we have

\begin{align}
\bigg|\frac{f(a_{1}+h,\ldots, a_{n-1})-f(a_{1},\ldots, a_{n-1})}{h}\bigg|  &\ll  |h|^{-1-\beta}|D(a_{1}+h,\ldots,a_{n-1})|^{M\alpha}\nonumber\\
 &= |h|^{M\alpha-1-\beta} \bigg|\frac{\partial D}{\partial x_{1}}(a_{1}+\theta h,\ldots, a_{n-1}) \bigg|^{M\alpha}\nonumber\\
\end{align}
for some $\theta\in (-1,1)$. By the continuity of $\frac{\partial D}{\partial x_{1}}$ and $M\alpha-1-\beta>0$,

\begin{equation}
\frac{\partial f}{\partial x_{1}}(a_{1},\ldots, a_{n-1}) = \lim_{h\to 0}\frac{f(a_{1}+h,\ldots, a_{n-1})-f(a_{1},\ldots, a_{n-1})}{h}=0.
\end{equation}
Similarly, $\frac{\partial f}{\partial x_{i}}(a_{1},\ldots, a_{n-1})=0$ for $1 \leq i \leq n-1$. By carrying out the above argument inductively, we have all of the partial derivatives at $x=a$ exists. Similar argument holds for the points on $T\setminus T_{\text{reg}}$ which are not the singularities of the real orbital integral. Therefore, $f$ is a smooth function on $\mathbb{R}^{n-1}$ and this completes the proof.
\end{proof}

%%%%%%%%%%%%%%%%%%%%%%%%%%%%%%%%%%%%%%%%%%%%%%%%%%%%%%%%%%%%%%%%%%%%%%%%%%%%%%%%%%%%%%%%%%%%%%%%%%%%%%%%%%%%%%%%%%%%%%%%%%%%%%%%%%%%%%%%%%%%
%%%%%%%%%%%%%%%%%%%%%%%%%%%%%%%%%%%%%%%%%%%%%%%%%%%%%%%%%%%%%%%%%%%%%%%%%%%%%%%%%%%%%%%%%%%%%%%%%%%%%%%%%%%%%%%%%%%%%%%%%%%%%%%%%%%%%%%%%%%%
%%%%%%%%%%%%%%%%%%%%%%%%%%%%%%%%%%%%%%%%%%%%%%%%%%%%%%%%%%%%%%%%%%%%%%%%%%%%%%%%%%%%%%%%%%%%%%%%%%%%%%%%%%%%%%%%%%%%%%%%%%%%%%%%%%%%%%%%%%%%
% Section: Future directions
%%%%%%%%%%%%%%%%%%%%%%%%%%%%%%%%%%%%%%%%%%%%%%%%%%%%%%%%%%%%%%%%%%%%%%%%%%%%%%%%%%%%%%%%%%%%%%%%%%%%%%%%%%%%%%%%%%%%%%%%%%%%%%%%%%%%%%%%%%%%
%%%%%%%%%%%%%%%%%%%%%%%%%%%%%%%%%%%%%%%%%%%%%%%%%%%%%%%%%%%%%%%%%%%%%%%%%%%%%%%%%%%%%%%%%%%%%%%%%%%%%%%%%%%%%%%%%%%%%%%%%%%%%%%%%%%%%%%%%%%%
%%%%%%%%%%%%%%%%%%%%%%%%%%%%%%%%%%%%%%%%%%%%%%%%%%%%%%%%%%%%%%%%%%%%%%%%%%%%%%%%%%%%%%%%%%%%%%%%%%%%%%%%%%%%%%%%%%%%%%%%%%%%%%%%%%%%%%%%%%%%
\section{Remarks on general reductive groups}\label{Directions}

Now let $G$ be a reductive group over $\Q$. In this section we indicate how the preceding analysis can be extended to general $G$, though for general $G$ we only consider unstable elliptic orbital integrals. As before, $\gamma$ will be an elliptic element of $G(\Q)$, so that $G_\gamma(\Q)$ is a torus. Based on work of Ono \cite{Ono}, Shyr deduced a class number relation for tori. We briefly describe this, and refer to \cite[\S3]{Shy} for details.

\subsection{}
Consider $T=G_\gamma$ as an algebraic torus over $\Q$, and let $\hat{T}=\text{Hom}(T,\G_m)$ be the $\Z$-module of rational characters of $T$. The torus $T$ splits over a finite separable extension $K$ of $\Q$, and $\Gamma=\text{Gal}(K/\Q)$ acts on $\hat{T}$. Then $\hat{T}$ becomes a free $\Gamma$-module with rank $r=\dim(T)$, and denote by $\chi_T$ the character of the $\Gamma$-module. The character decomposes into
\be
\chi_T = \sum_{i=0}^h m_i \chi_i
\ee
for some integer $h$. Here $\chi_i$ are irreducible characters of $G$ with $\chi_0$ the principal character, and $m_i$ the multiplicity, whereby $m_0=r$. It follows then that the Artin $L$-function factorizes as
\be
L(s,\chi_T) = \zeta(s)^r\prod_{i=1}^h L(s,\chi_i)^{m_i}.
\ee
Moreover, for $i\geq2$, $L(1,\chi_i)$ is nonzero, so that the value
\be
\rho_T\ :=\ \lim_{s\to 1}(s-1)^rL(s,\chi_T) = \prod^h_{i=1} L(1,\chi_i)
\ee
is finite and nonzero, and is called the quasi-residue of $T$ over $\Q$. By \cite{Ono}, it is independent of choice of splitting field.

Now, choosing canonical Haar measures related to the Tamagawa numbers, Shyr obtains the relation
\be
\rho_T =  \frac{h_T R_T}{\tau_T w_T D_T^{1/2}}
\ee
where $\tau_T$ is the Tamagawa number of $T$, and the other $h_T,R_T,w_T,$ and $D_T$ are arithmetic invariants of $T$ defined analogous to those appearing in Dirichlet's class number formula \eqref{eqn cnf}.

Then one may proceed as in Section \ref{prelim}, and in particular, using \eqref{CNFA} to write the volume term as the value at $1$ of an Artin $L$-function, and apply the approximate functional equation. Then by similar estimates in Section \ref{Smoothing} the singularities of the real orbital integral may be smoothed.

\begin{rem}
In the case of $G=GL(n)$ the element $\gamma$ is elliptic and defines a degree $n$ extension $E$ over $\Q$, and the torus is simply the Weil restriction Res$_{E/\Q}(\G_m),$ split by $K$. By the remark following \cite[Theorem 1]{Shy}, one indeed recovers
\be
\rho_T = \res_{s=1}\zeta_E(s),
\ee
recovering the original case, and in particular the analytic class number formula.
\end{rem}

%\appendix

%%%%%%%%%%%%%%%%%%%%%%%%%%%%%%%%%%%%%%%%%%%%%%%%%%%%%%%%%%%%%%%%%%%%%%%%%%%%%%%%%%%%%%%%%%%%%%%%%%%%%%%%%%%%%%%%%%%%%%%%%%%%%%%%%%%%%%%%%%%%
%%%%%%%%%%%%%%%%%%%%%%%%%%%%%%%%%%%%%%%%%%%%%%%%%%%%%%%%%%%%%%%%%%%%%%%%%%%%%%%%%%%%%%%%%%%%%%%%%%%%%%%%%%%%%%%%%%%%%%%%%%%%%%%%%%%%%%%%%%%%
%%%%%%%%%%%%%%%%%%%%%%%%%%%%%%%%%%%%%%%%%%%%%%%%%%%%%%%%%%%%%%%%%%%%%%%%%%%%%%%%%%%%%%%%%%%%%%%%%%%%%%%%%%%%%%%%%%%%%%%%%%%%%%%%%%%%%%%%%%%%
% Appendix
%%%%%%%%%%%%%%%%%%%%%%%%%%%%%%%%%%%%%%%%%%%%%%%%%%%%%%%%%%%%%%%%%%%%%%%%%%%%%%%%%%%%%%%%%%%%%%%%%%%%%%%%%%%%%%%%%%%%%%%%%%%%%%%%%%%%%%%%%%%%
%%%%%%%%%%%%%%%%%%%%%%%%%%%%%%%%%%%%%%%%%%%%%%%%%%%%%%%%%%%%%%%%%%%%%%%%%%%%%%%%%%%%%%%%%%%%%%%%%%%%%%%%%%%%%%%%%%%%%%%%%%%%%%%%%%%%%%%%%%%%
%%%%%%%%%%%%%%%%%%%%%%%%%%%%%%%%%%%%%%%%%%%%%%%%%%%%%%%%%%%%%%%%%%%%%%%%%%%%%%%%%%%%%%%%%%%%%%%%%%%%%%%%%%%%%%%%%%%%%%%%%%%%%%%%%%%%%%%%%%%%

\section{Applications and obstructions}\label{sec6}

In this section, we discuss how the main theorem generalizes Altu\u{g}'s analysis to $GL(n)$, and the briefly describe the remaining issues to be overcome in order to apply Poisson summation in this case: absorbing the $p$-adic orbital into the volume factor, and controlling the Fourier transforms in the dual sum.

\subsection{The $p$-adic orbital integrals}

The main obstruction relates to the $p$-adic orbital integrals. By Equation (59) in \cite[\S2.5]{Lan1}, the product of the $p$-adic orbital integrals for $GL(2)$ can be expressed using the Kronecker symbol,

\be
\label{padicorbs}
\prod_q \int_{G_{\gamma}(\Q_{q}) \backslash G(\Q_{q})} f_{q}(g^{-1} \gamma g) d\bar{g}_{q}
 = \
\sum_{f|s_{\gamma}} f \prod_{q|f} \left(1-\left(\frac{D}{q}\right)\frac{1}{q}\right).
\ee
Then by \cite[\S 2.2.2]{Alt1}, this can be combined with global volume factor by a change of variables to give

\begin{align}
&\sqrt{|D|} L\left(1,\left(\frac{D}{\cdot}\right)\right) \sum_{f|s} f \prod_{q|f} \left(1-\left(\frac{D}{q}\right)\frac{1}{q}\right) =  \sum_{f|s}\frac{1}{f}L\left(1,\left(\frac{(m^2-N)/f^2}{\cdot}\right)\right),
\end{align}
where $s^2D=m^2-N$ and $D$ a discriminant of a quadratic number field. Then, the sum of $L$-functions is itself shown to be an $L$-function,
\be
\label{lfunsum}
L(1,m^2-N)=\sum^{'}\frac{1}{f}L\left(1,\left(\frac{(m^2-N)/f^2}{\cdot}\right)\right)
\ee
where the sum is taken over those $f$ satisfying $f^2 \ | \ M^2-N$ and $(m^2-N)/f^2 \equiv 0, 1 \ (\bmod\, 4)$, and the approximate functional equation is applied to this final form.

In order to generalize this to other groups, one would need an expression for the $p$-adic orbital integrals related to the Artin representation $\chi_T$ as in Section \ref{Directions}. Unfortunately, for general groups a closed formula is not known for these integrals, though it is interesting to note that by \cite{Ngo}, and not to mention \cite{Hal}, one knows that evaluating such $p$-adic orbital integrals is closely related to counting points on varieties over finite fields.

\begin{rem}
In the case of $GL(3)$ the $p$-adic orbital integrals for the unit element of the Hecke algebra have in fact been computed explicitly by Kottwitz \cite[p.661]{Kot}. If the elliptic element $\gamma=\alpha+\pi^n\beta$ generates an unramified cubic extension, one has
\be
\frac{p^{3n+1}(p+1)(p^2+p+1)-3p^{2n}(p^2+p+1) +3}{(p-1)^2(p+1)},
\ee
while for a ramified cubic extension one has
\be
\frac{p^{3n+1}(p+1)p^{1+\text{val}(\beta)}-p^{2n}(p^2+(p+1)p^{2\text{val}(\beta)})+1}{(p-1)^2(p+1)},
\ee
where $\mathrm{val}(\beta)=1$ or 2, but it does not seem clear what the analogous identity for \eqref{padicorbs} might be.
\end{rem}

\subsection{Poisson summation}

Supposing for the moment that we had such an expression like (\ref{lfunsum}) for $G=GL(n)$, whereby the elliptic orbital integral
\[
\sum_{\gamma \text{ ell}} \operatorname{meas}(\gamma)  \int_{G_{\gamma}(\A) \backslash G(\A)} f(g^{-1} \gamma g) dg,
\]
which we expressed as
\be
\sum_{\pm p^{k}} \sideset{}{'}\sum
\frac{1}{|s_{\gamma}|}L(1,\sigma_{E})\theta_{\infty}^\pm(\gamma) \prod_{q}\operatorname{Orb}(f_{q};\gamma),
\ee
can be simplified into
\be
\sum_{\pm p^{k}} \sideset{}{'}\sum\frac{1}{|s_{\gamma}|}L(1,\sigma'_{E})\theta_{\infty}^\pm(\gamma)
\ee
where $\sigma'_E$ is another degree $n-1$ Artin representation related to the original $\sigma_E$. Then we could apply the approximate functional equation to $L(1,\sigma_E')$, and thus Theorem \ref{smoothe} to the resulting expression to allow for Poisson summation.

Next, we formally complete the inner sum as in \cite{Alt1} by adding and subtracting the missing terms, the new subtracted terms either vanishing in the limit (\ref{rtrace}) or need to be treated separately.\footnote{We thank the reviewer for pointing this out to us.} In order for the Poisson summation formula over the inner sum on $\Z^{n-1}$ to be valid, we must control the Fourier transform of the dual sum. For this, we would require a more precise characterization of the singularities of the orbital integrals than that discussed in Section \ref{sing}.

Besides the main obstacles anticipated and addressed in \cite{Alt1, Alt2} for applying the Poisson summation (i.e., absolute convergence of $L$-series, summation over incomplete lattice and singularities contributed by the real orbital integral),there is one extra problem appearing in the higher rank cases which we would like to highlight:

We first state the version of equation (\ref{coeff}) in \cite{Alt1, Alt2} with the missing terms correspond to $D_{\gamma}$ being perfect squares being added back in is:
\begin{align}\label{alt}
\sum_{\mp}\sum_{f=1}^{\infty} \frac{1}{f} \sum_{\ell=1}^{\infty} \frac{1}{\ell} \sideset{}{'}\sum_{\substack{m\in\mathbb{Z}\\ f^2  | m^2\pm 4p^k }}\left(\frac{(m^2\pm 4p^k)/f^2}{\ell}\right)\theta_{\infty}^{\mp}\left(\frac{m}{2p^{k/2}}\right)\times\nonumber\\ \left[F\left(\frac{\ell f^2}{|m^2\pm 4p^k|^{\alpha}}\right)+\frac{\ell f^2}{\sqrt{|m^2\pm 4p^k|}} H\left(\frac{\ell f^2}{|m^2\pm 4p^k|^{1-\alpha}}\right)\right],
\end{align}

\noindent where $F$ and $H$ are some Schwartz function, $\alpha\in (0,1)$, and $'$ in the summation sign means we only sum over $m$ such that $(m^2\pm 4p^{k})/f^2 \equiv 0,1\  (\bmod\, 4)$.

Not only are the coefficients of the $L$-series explicitly given by the Kronecker symbol, they are also \textit{periodic}. If we want to apply Poisson summation to the $m$-sum, the summands have to be defined on $\R$, smooth and with good decay. Of course, the Kronecker symbol is only defined for $m\in \mathbb{Z}$. Altu\u{g} crucially uses the periodicity to split  the $m$-sum, bringing the Kronecker symbol out of the sum: \footnote{\ There is a slight mistake in \cite[pp. 1807]{Alt1}. After taking the Kronecker symbol out of the $m$-sum, it should no longer depend on $m$. It should be $\left(\frac{(a^2\pm 4p^k)/f^2}{\ell}\right)$ instead, as in equation (\ref{correct}). Also, for the same equation in \cite[p. 1796]{Alt1},  $'$ on the summation sign should refer to summing over those $f$ such that $f^2 | m^2 \pm 4p^k$ and $(m^2\pm 4p^k)/f^2 \equiv 0, 1 \ (\bmod\, 4) $ instead of $f| m^2 \pm 4p^k$ and $(m^2\pm 4p^k)/f^k \equiv 0, 1 \ (\bmod\, 4)$.}

\begin{align}\label{correct}
&\sum_{\substack{a \bmod 4\ell f^2\\ a^2 \pm 4p^k \equiv 0 \ \bmod\, f^2\\ (a^2\pm 4p^k)/f^2 \equiv 0,1 \bmod\, 4}} \sum_{\substack{m\in \Z \\ m \equiv a \bmod\, 4\ell f^2}} \left(\frac{(m^2\pm 4p^k)/f^2}{\ell}\right) \theta_{\infty}^{\mp}\left(\frac{m}{2p^{k/2}}\right)G_{\ell, f}^{\pm}(m) \nonumber\\
= & \sum_{\substack{a \bmod 4\ell f^2\\ a^2 \pm 4p^k \equiv 0 \ \bmod\, f^2\\ (a^2\pm 4p^k)/f^2 \equiv 0,1 \bmod\, 4}} \left(\frac{(a^2\pm 4p^k)/f^2}{\ell}\right) \sum_{\substack{m\in \Z \\ m \equiv a \bmod\, 4\ell f^2}}  \theta_{\infty}^{\mp}\left(\frac{m}{2p^{k/2}}\right)G_{\ell, f}^{\pm}(m),
\end{align}
where $G_{\ell, f}^{\pm}(m)$ denotes the bracketed term in the $m$-sum of equation (\ref{alt}). But in general for $GL(n), n\ge 3$, this only works for elliptic $\gamma$ defining an abelian extension, which is a particularly special case.

%%%%%%%%%%%%%%%%%%%%%%%%%%%%%%%%%%%%%%%%%%%%%%%%%%%%%%%%%%%%
% Bibliography
%%%%%%%%%%%%%%%%%%%%%%%%%%%%%%%%%%%%%%%%%%%%%%%%%%%%%%%%%%%%

\ \\

\begin{thebibliography}{KonS8}

\bibitem[Alt1]{Alt1}
S. A. Altu\u{g}, \emph{Beyond Endoscopy via the Trace Formula - I
Poisson Summation and Contributions of Special Representations}, Composito Mathematica \textbf{151} (2015), 1791--1820.

\bibitem[Alt2]{Alt2}
S. A. Altu\u{g}, \emph{Beyond Endoscopy via the Trace Formula} (Ph. D. Thesis),
available at \bburl{http://arks.princeton.edu/ark:/88435/dsp01b5644r615}.

\bibitem[Art]{Art}
J. Arthur, \emph{Problems Beyond Endoscopy}, Roger Howe volume, \textit{preprint},
available at \bburl{http://www.math.toronto.edu/arthur/pdf/Arthur.pdf}.

\bibitem[Bou]{Bou}
A. Bouaziz, \emph{Formule d'inversion des int\'egrales orbitales sur les groupes de Lie r\'eductifs,} J. Funct. Anal. {\bf 134} (1995), no.~1, 100--182. MR1359924

\bibitem[FLN]{FLN}
E. Frenkel, R. P. Langlands and B. C. Ng\^{o},
\emph{La formule des traces et la functorialit\'e. Le debut dun Programme},
Ann. Sci. Math. Qu\'ebec 34 (2010), 199--243.

\bibitem[Hal]{Hal}
T. Hales, \emph{Hyperelliptic curves and harmonic analysis (why harmonic analysis on reductive p-adic groups is not elementary)},
Representation theory and analysis on homogeneous spaces (New Brunswick, NJ, 1993), 137--169,
Contemp. Math. \textbf{177}, Amer. Math. Soc., Providence, RI, 1994.

\bibitem[HC1]{HC1}
Harish-Chandra, \emph{Harmonic analysis on real reductive groups. I. The theory of the constant term}, J. Functional Analysis \textbf{19} (1975), 104--204.

\bibitem[HC2]{HC2}
Harish-Chandra, \emph{Invariant eigendistributions on a semisimple Lie group}, Trans. Amer. Math. Soc. \textbf{119} (1965), 457--508.

\bibitem[IK]{IK}
H. Iwaniec and E. Kowalski, \textit{Analytic number theory}, AMS Colloquium Publications, vol. 53, 2004.

\bibitem[Kot1]{Kot}
R. Kottwitz, \emph{Unstable orbital integrals on $SL(3)$}, Duke Math. J. (3) 48 (1981), 649--664.

\bibitem[Kot2]{Kot2}
R. Kottwitz, \emph{Harmonic analysis on reductive p-adic groups and Lie algebras},
in Harmonic analysis, the trace formula, and Shimura varieties,
pp. 393-522, Clay Math. Proc. \textbf{4}, Amer. Math. Soc., Providence, RI, 2005.



\bibitem[Lan1]{Lan1}
R. P. Langlands, \emph{Beyond Endoscopy}, in Contributions to automorphic forms, geometry and number theory, pp. 611-697, Johns Hopkins University Press, Baltimore, MD, 2004.

%\bibitem[Lan2]{Lan2}R. P. Langlands, \emph{Singularit\'es et Transfert}, available at: \bburl{http://publications.ias.edu/sites/default/files/transfert_1.pdf}.

\bibitem[Mur]{Mur}
Murty, M. Ram, Murty, V. Kumar, \emph{Non-vanishing of L-Functions and Applications}, Birkh\"{a}user Basel (1997).

\bibitem[Neu]{Neu}
Neukirch, J\"{u}rgen, \emph{Algebraic Number Theory}, Springer-Velag Berlin Heidelberg (1999).




\bibitem[Ng\^{o}]{Ngo}
B. C. Ng\^{o}, \emph{Le lemme fondamental pour les algebres de Lie}, Publ. Math. Inst. Hautes \'Etudes Sci. \textbf{111} (2010), 1--169.

%\bibitem[Mok]{Mok}C. P. Mok, \emph{A weak form of beyond endoscopic decomposition for the stable trace formula of odd orthogonal groups}, \textit{preprint},
available at \bburl{https://arxiv.org/abs/1608.03330}.

\bibitem[Ono]{Ono}
T. Ono, \emph{Arithmetic of algebraic tori}, Ann. of Math. (2) \textbf{74} (1961), 101--139.

%\bibitem[RV]{RV} D. Ramakrishnan, R. J. Valenza,  \textit{Fourier analysis on number fields}, Graduate Text in Mathematics No. \textbf{186}, Springer (1999).

\bibitem[Sh1]{Sh1}
D. Shelstad, \emph{Characters and inner forms of a quasi-split group over ${\bf R}$,} Compositio Math. {\bf 39} (1979), no.~1, 11--45. MR0539000

\bibitem[Sh2]{Sh2}
D. Shelstad, Orbital integrals, endoscopic groups and $L$-indistinguishability for real groups, in {\it Conference on automorphic theory (Dijon, 1981)}, 135--219, Publ. Math. Univ. Paris VII, 15, Univ. Paris VII, Paris. MR0723184

\bibitem[Shy]{Shy}
J. M. Shyr, \emph{On some class number relations of algebraic tori}, Michigan Math. J. \textbf{24} (1977), no. 3, 365--377.

\bibitem[T]{T}
J. T. Tate, \emph{Fourier analysis in number fields, and Hecke's zeta-functions}, in Algebraic Number Theory (Proc. Instructional Conf., Brighton, 1965), 305--347, Thompson, Washington, DC. %MR0217026

\end{thebibliography}
\end{document}